%% file: paper-oddness-amc-revision.tex
\documentclass{amcjoucc}

\usepackage[english]{babel}
\usepackage{amsmath, amssymb, amsthm}
\usepackage{amsfonts}
\usepackage{latexsym}
\usepackage{xcolor}
\usepackage{enumerate}
\usepackage{graphicx}

\usepackage{thm-restate}
\usepackage{thmtools}

\usepackage{hyperref}
\usepackage[small,bf,hang]{caption} 

\def\Col{\operatorname{Col}}
\newcommand{\claimproofend}{\hspace*{.1mm}\hspace{\fill}}

\graphicspath{{figures/}}

\begin{document}

\begin{frontmatter}

\titledata{Smallest snarks with oddness 4 and \\cyclic connectivity 4 have order 44}{}           

\authordata{Jan Goedgebeur}            
{Department of Applied Mathematics, Computer Science \&
Statistics, Ghent University, Krijgslaan 281-S9, 9000 Ghent, Belgium\\
Computer Science Department, University of Mons, \\Place du Parc 20, 7000 Mons, Belgium}    
{jan.goedgebeur@ugent.be}                     
{}

\authordata{Edita M\' a\v cajov\' a}            
{Department of Computer Science, Comenius University, 842 48 Bratislava, Slovakia}    
{macajova@dcs.fmph.uniba.sk}
{}                                       

\authordata{Martin \v Skoviera}            
{Department of Computer Science, Comenius University, 842 48 Bratislava, Slovakia}    
{skoviera@dcs.fmph.uniba.sk}
{}                                       


\begin{abstract}

The family of snarks -- connected bridgeless cubic graphs that
cannot be $3$-edge-colour\-ed -- is well-known as a potential
source of counterexamples to several important and
long-standing conjectures in graph theory. These include the
cycle double cover conjecture, Tutte's 5-flow conjecture,
Fulkerson's conjecture, and several others. One way of
approaching these conjectures is through the study of
structural properties of snarks and construction of small
examples with given properties. In this paper we deal with the
problem of determining the smallest order of a nontrivial snark
(that is, one which is cyclically $4$-edge-connected and has
girth at least $5$) of oddness at least $4$. Using a
combination of structural analysis with extensive computations
we prove that the smallest order of a snark with oddness at
least $4$ and cyclic connectivity~$4$ is $44$. Formerly it was
known that such a snark must have at least 38 vertices [J.\
Combin.\ Theory Ser.\ B 103 (2013), 468--488] and one such
snark on 44 vertices was constructed by Lukot\!'ka et al.
[Electron.\ J.\ Combin.\ 22 (2015), $\#$P1.51]. The proof
requires determining all cyclically 4-edge-connected snarks on
36 vertices, which extends the previously compiled list of all
such snarks up to 34 vertices [J.\ Combin.\ Theory Ser.\ B,
loc.\ cit.]. As a by-product, we use this new list to test the
validity of several conjectures where snarks can be smallest
counterexamples.
\end{abstract}

\end{frontmatter}

\section{Introduction}\label{sec:intro}

Snarks are an interesting, important, but somewhat mysterious
family of cubic graphs whose characteristic property is that
their edges cannot be properly coloured with three colours.
Very little is known about the nature of snarks because the
reasons which cause the absence of $3$-edge-colourability in
cubic graphs are not well understood. Snarks are also difficult
to find because almost all cubic graphs are hamiltonian and
hence $3$-edge-colourable~\cite{RW}. On the other hand,
deciding whether a cubic graph is $3$-edge-colourable or not is
NP-complete~\cite{holyer}, implying that the family of snarks
is sufficiently rich.

The importance of snarks resides mainly in the fact that many
difficult conjectures in graph theory, such as Tutte's $5$-flow
conjecture or the cycle double cover conjecture, would be
proved in general if they could be established for snarks
\cite{J85, J88}. While most of these problems are trivial for
3-edge-colourable graphs, and exceedingly difficult for snarks
in general, they often become tractable for snarks that are in
a certain sense close to being $3$-edge-colourable.

There exist a number of measures of uncolourability of cubic
graphs (see~\cite{FMS-survey} for a recent survey). Among them,
the smallest number of odd circuits in a $2$-factor of a cubic
graph, known as \textit{oddness}, has received the widest
attention.  Note that the oddness of a cubic graph is an even
integer which equals zero precisely when the graph is
$3$-edge-colourable. It is known, for example, that the
$5$-flow conjecture and the Fan-Raspaud conjecture are true for
cubic graphs of oddness at most two~\cite{J88, MS:4-lines},
while the cycle double cover conjecture is known to hold for
cubic graphs of oddness at most $4$~\cite{hg,hk}. Snarks with
large oddness thus still remain potential counterexamples to
these conjectures and therefore merit further study.

\begin{figure}[htbp]
	\centering
	\includegraphics[width=0.4\textwidth]{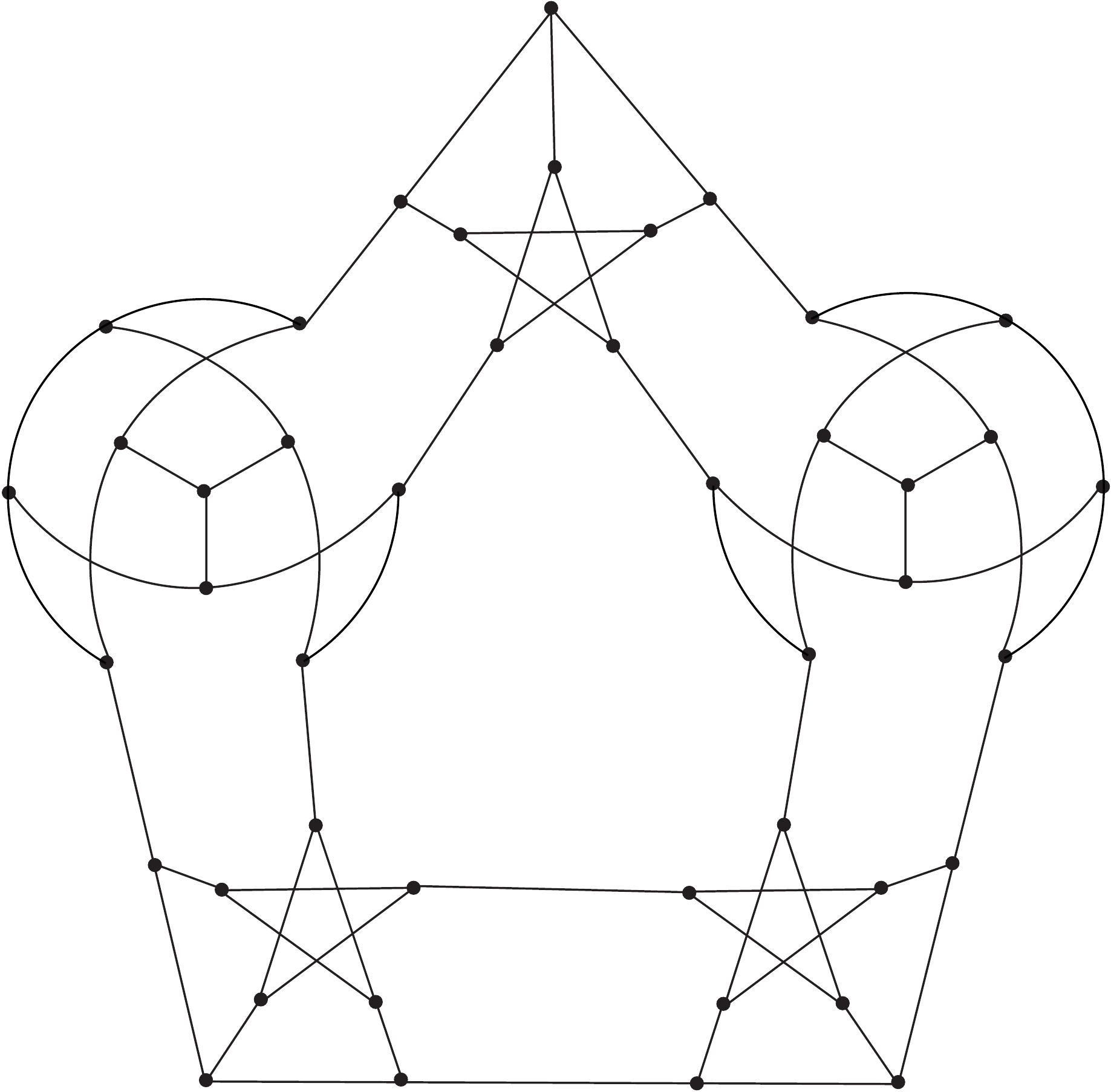}
  \caption{The smallest known nontrivial snark with oddness $\ge 4$.}
  \label{fig:LMMS44}
\end{figure}

Several authors have provided constructions of infinite
families of snarks with increasing oddness, see, for example,
\cite{hlund, Kochol:sup, LMMS, Steffen:meas}. Most of them
focus on snarks with cyclic connectivity at least $4$ and girth
at least $5$, because snarks that lack these two properties can
be easily reduced to smaller snarks. We call such snarks
\textit{nontrivial}. All currently available constructions
indicate that snarks of oddness greater than $2$ are extremely
rare. From~\cite[Observation~4.10]{BrGHM} it follows that there
exist no nontrivial snarks of oddness greater than $2$ on up to
36 vertices. The smallest known example of a nontrivial snark
with oddness at least $4$ has $44$ vertices and its oddness
equals $4$. It was constructed by  Lukot'ka et al.
in~\cite{LMMS}, superseding an earlier construction of
H\"agglund \cite{hlund} on 46 vertices; it is shown in
Figure~\ref{fig:LMMS44} in a form different from the one
displayed in \cite{LMMS}. In~\cite[Theorem~12]{LMMS} it is also
shown that if we allow trivial snarks, the smallest one with
oddness greater than $2$ has $28$ vertices and oddness $4$. As
explained in~\cite{G, LMMS-err}, there are exactly three such
snarks, one with cyclic connectivity $3$ and two with cyclic
connectivity $2$. (The latter result rectifies the false claim
made in~\cite{LMMS} that there are only two snarks of oddness 4
on 28 vertices.)

The aim of the present paper is to prove the following result.

\begin{restatable}{thm}{mainthm}\label{thm:main}
The smallest number of vertices of a snark with cyclic
connectivity $4$ and oddness at least $4$ is $44$. The girth of
each such snark is at least $5$.
\end{restatable}

This theorem bridges the gap between the order $36$ up to which
all nontrivial snarks have been generated (and none of oddness
greater than $2$ was found~\cite{BrGHM}) and the order $44$
where an example of oddness $4$ has been constructed
\cite{LMMS}. Since generating all nontrivial snarks beyond $36$
vertices seems currently infeasible, it would be hardly
possible to find a smallest nontrivial snark with oddness at
least $4$ by employing computational force alone. On the other
hand, the current state-of-the-art in the area of snarks, with
constructions significantly prevailing over structural
theorems, does not provide sufficient tools for a purely
theoretical proof of our theorem. Our proof is therefore an
inevitable combination of structural analysis of snarks with
computations.

The proof consists of two steps. First we prove that every
snark with oddness at least~$4$, cyclic connectivity~$4$, and
minimum number of vertices can be decomposed into two smaller
cyclically $4$-edge-connected snarks $G_1$ and $G_2$ by
removing a cycle-separating $4$-edge-cut, adding at most two
vertices to each of the components, and by restoring
$3$-regularity. Conversely, every such snark arises from two
smaller cyclically $4$-edge-connected snarks $G_1$ and $G_2$ by
the reverse process. In the second step of the proof we
computationally verify that no combination of $G_1$ and $G_2$
can  result in a cyclically $4$-edge-connected snark of oddness
at least $4$ on fewer than $44$ vertices. This requires
checking all suitable pairs of cyclically $4$-edge-connected
snarks on up to $36$ vertices, including those that contain
$4$-cycles. Such snarks have been previously generated only up
to order $34$~\cite{BrGHM}, which is why we had to additionally
generate all cyclically $4$-edge-connected snarks on $36$
vertices containing a $4$-cycle. This took about 80 CPU years
and yielded exactly $404~899~916$ cyclically $4$-edge-connected
snarks.

It is important to realise that Theorem~\ref{thm:main} does not
yet determine the order of a smallest nontrivial snark with
oddness at least $4$. The reason is that it does not exclude
the existence of cyclically $5$-connected snarks with oddness
at least $4$ on fewer than $44$ vertices. However, the smallest
currently known cyclically $5$-edge-connected snark with
oddness at least $4$ has $76$ vertices (see Steffen
\cite[Theorem~2.3]{Steffen:meas}), which indicates that
a cyclically $5$-edge-connected snark with oddness at least $4$ on fewer than 44 vertices
either does not exist or will be very difficult to find.

Our paper is organised as follows. Section~\ref{sec:prelim}
provides the necessary background material for the proof of
Theorem~\ref{thm:main} and for the results that precede it, in
particular for the decomposition theorems proved in
Section~\ref{sec:decomp}. In Section~\ref{sec:MainResult} we
employ these decomposition theorems to prove
Theorem~\ref{thm:main}. We further discuss this theorem in
Section~\ref{sec:remarks} where we also pose two related
problems. In the final section we report about the tests which
we have performed on the set of all cyclically
$4$-edge-connected snarks of order 36 concerning the validity
of several interesting conjectures in graph theory, such as the
dominating cycle conjecture, the total colouring conjecture,
and the Petersen colouring conjecture.

We will continue our investigation of the smallest snarks with
oddness at least $4$ and cyclic connectivity $4$ in the sequel
of this paper~\cite{GMS}. We will display a set of 31 such
snarks, analyse their properties, and prove that they constitute
the complete set of snarks with oddness at least $4$ and cyclic
connectivity $4$ on $44$ vertices.

\section{Preliminaries}\label{sec:prelim}

\textbf{2.1. Graphs and multipoles.} All graphs in this paper
are finite. For the sake of completeness, we have to permit
graphs containing multiple edges or loops, although these
features will in most cases be excluded by the imposed
connectivity or colouring restrictions.

Besides graphs we also consider graph-like structures, called
\textit{multipoles}, that may contain dangling edges and even
isolated edges. Multipoles serve as a convenient tool for
constructing larger graphs from smaller building blocks. They
also naturally arise as a result of severing one or several
edges of a graph, in particular edges forming an edge-cut. In
this paper all multipoles will be cubic (3-valent).

Every edge of a multipole has two ends and each end can, but
need not, be incident with a vertex. An edge which has both
ends incident with a vertex is called \textit{proper}. If one
end of an edge is incident with a vertex and the other is not,
then the edge is called a \textit{dangling edge} and, if
neither end of an edge is incident with a vertex, it is called
an \textit{isolated edge}. An end of an edge that is not
incident with a vertex is called a \textit{semiedge}. A
multipole with $k$ semiedges is called a \textit{$k$-pole}. Two
semiedges $s$ and $t$ of a multipole can be joined to produce
an edge $s*t$ connecting the end-vertices of the corresponding
dangling edges. Given two $k$-poles $M$ and $N$ with semiedges
$s_1, \ldots, s_k$ and $t_1,\ldots, t_k$, respectively, we
define their \textit{complete junction} $M*N$ to be the graph
obtained by performing the junctions $s_i*t_i$ for each
$i\in\{1,\ldots, k\}$. A \textit{partial junction} is defined
in a similar way except that a proper subset of semiedges of
$M$ is joined to semiedges of $N$. Partial junctions can be
used to construct larger multipoles from smaller ones. In
either case, whenever a junction of two multipoles is to be
performed, we assume that their semiedges are assigned a fixed
order. For a more detailed formal development of concepts
related to multipoles we refer the reader, for example,
to~\cite{Fiol:bool, MS:irred} or~\cite{ChS}.

\medskip\noindent
\textbf{2.2. Cyclic connectivity.} Let $G$ be a connected
graph. An \textit{edge-cut} of a graph~$G$, or just a
\textit{cut} for short, is any set $S$ of edges of $G$ such
that $G-S$ is disconnected. An edge-cut is said to be
\textit{trivial} if it consists of all edges incident with one
vertex, and \textit{nontrivial} otherwise. An important kind of
an edge-cut is a cocycle, which arises by taking a set of
vertices or an induced subgraph $H$ of $G$ and letting $S$ to
be the set $\delta_G(H)$ of all edges with exactly one end in
$H$. We omit the subscript $G$ whenever $G$ is clear from the
context.

An edge-cut is said to be \textit{cycle-separating} if at least
two components of $G-S$ contain cycles. We say that a connected
graph $G$ is \textit{cyclically $k$-edge-connected} if no set
of fewer than $k$ edges is cycle-separating in $G$. The
\textit{cyclic connectivity} of $G$, denoted by $\zeta(G)$,  is
the largest number $k\le\beta(G)$, where
$\beta(G)=|E(G)|-|V(G)|+1$ is the cycle rank of $G$, for which
$G$ is cyclically $k$-connected (cf.~\cite{NS:cc,R}).

It is not difficult to see that for a cubic graph $G$ with
$\zeta(G)\le 3$ the value $\zeta(G)$ coincides with the usual
vertex-connectivity or edge-connectivity of $G$. Thus cyclic
connectivity in cubic graphs is a natural extension of the
common versions of connectivity (which unlike cyclic
connectivity are bounded above by $3$). Another useful
observation is that the value of cyclic connectivity remains
invariant under subdivisions and adjoining new vertices of
degree~$1$.

The following well-known result \cite{NS:cc,R} relates
$\zeta(G)$ to the length of a shortest cycle in $G$, denoted by
$g(G)$ and called the \textit{girth} of $G$.

\begin{proposition}\label{prop:girth}
For every connected cubic graph $G$ we have $\zeta(G)\le g(G)$.
\end{proposition}

Let us observe that in a connected cubic graph every edge-cut
$S$ consisting of independent edges is cycle-separating: indeed
the minimum valency of $G-S$ is $2$, so each component of $G-S$
contains a cycle. Conversely, a cycle-separating edge-cut of
minimum size is easily seen to be independent; moreover, $G-S$
has precisely two components, called \textit{cyclic parts} or
\textit{fragments}. A fragment minimal under inclusion will be
called an \textit{atom}. A \textit{nontrivial atom} is any
atom different from a shortest cycle.

The following two propositions provide useful tools in handling
cyclic connectivity. The first of them follows easily by
mathematical induction. For the latter we refer the reader to
\cite[Proposition~4 and Theorem~11]{NS:cc}.

\begin{lemma}\label{lemma:pvmultipol}
Let $H$ be a connected acyclic subgraph of a cubic graph
separated from the rest by a $k$-edge-cut. Then $H$ has $k-2$
vertices.
\end{lemma}

\begin{proposition}\label{prop:cc-properties}
Let $G$ be a connected cubic graph. The following statements
hold:
\begin{itemize}
\item[{\rm (i)}] Every fragment of $G$ is connected, and
    every atom is $2$-connected. Moreover, if $\zeta(G)\ge
    3$, then every fragment is $2$-connected.
\item[{\rm (ii)}] If $A$ is a nontrivial atom of $G$, then
    $\zeta(A)>\zeta(G)/2$.
\end{itemize}
\end{proposition}

In the present paper we focus on cyclically $4$-edge-connected
cubic graphs, in particular on those with cyclic connectivity
exactly $4$. From the results mentioned earlier it follows that
a cyclically $4$-edge-connected cubic graph has no bridges and
no $2$-edge-cuts. Furthermore, every $3$-edge-cut separates a
single vertex, and every $4$-edge-cut which is not
cycle-separating consists of the four edges adjacent to some
edge.

An important method of constructing cyclically
$4$-edge-connected cubic graphs from smaller ones applies the
following operation which we call an I-\textit{extension}. In a cubic
graph $G$ take two edges $e$ and $f$, subdivide each of $e$ and
$f$ with a new vertex $v_e$ and $v_f$, respectively, and by add
a new edge between $v_e$ and $v_f$. The resulting graph,
denoted by $G(e,f)$ is said to be obtained by an
I-\textit{extension} across $e$ and $f$. It is not difficult to
see that if $G$ is cyclically $4$-edge-connected and $e$ and
$f$ are non-adjacent edges of $G$, then so is $G{(e,f)}$.

A
well-known theorem of Fontet~\cite{Font} and Wormald~\cite{W}
states that all cyclically $4$-edge-connected cubic graphs can
be obtained from the complete graph $K_4$ and the cube $Q_3$ by
repeatedly applying I-extensions to pairs of non-adjacent
edges. However, I-extensions are also useful for constructing
cubic graphs in general. For example, in~\cite{BrGM} all
connected cubic graphs up to 32 vertices have been generated by
using I-extensions as main construction operation.

For more information on cyclic connectivity the reader may wish
to consult~\cite{NS:cc}.

\medskip\noindent \textbf{2.3. Edge-colourings.} A
$k$-\textit{edge-colouring} of a graph $G$ is a mapping
$\phi\colon E(G)\to\mathbf{C}$ where $\mathbf{C}$ is a set of
$k$ colours. If all pairs of adjacent edges receive distinct
colours, $\phi$ is said to be \textit{proper}; otherwise it is
called \textit{improper}. Graphs with loops do not admit proper
edge-colourings because of the self-adjacency of loops. Since
we are mainly interested in proper colourings, the adjective
``proper'' will usually be dropped. For multipoles,
edge-colourings are defined similarly; that is to say, each
edge receives a colour irrespectively of the fact whether it
is, or it is not, incident with a vertex.

The result of Shannon~\cite{Sh} implies that every loopless
cubic graph, and hence every loopless cubic multipole, can be
properly coloured with four colours, see also~\cite{J}. In the
study of snarks it is often convenient to take the set of
colours $\mathbf{C}$ to be the set
$\mathbb{Z}_2\times\mathbb{Z}_2=\{(0,0),(0,1),(1,0),(1,1)\}$
where $(0,0)$, $(0,1)$, $(1,0)$, and $(1,1)$ are identified
with $0$, $1$, $2$, and $3$, respectively. We say that a
multipole is \textit{colourable} if it admits a
$3$-edge-colouring and \textit{uncolourable} otherwise. For a
$3$-edge-colouring of a cubic graph or a cubic multipole we use
the colour-set $\mathbf{C}=\{1,2,3\}$ because such a colouring
is in fact a nowhere-zero
$\mathbb{Z}_2\times\mathbb{Z}_2$-flow. This means that for
every vertex $v$ the sum of colours incident with $v$, the
\textit{outflow} at $v$, equals $0$ in
$\mathbb{Z}_2\times\mathbb{Z}_2$. The following fundamental
result~\cite{Bla, D} is a direct consequence of this fact.

\begin{theorem}\label{lemma:parity} {\rm (Parity Lemma)}
Let $M$ be a $k$-pole endowed with a proper
$3$-edge-colouring with colours $1$, $2$, and $3$. If the set
of all semi-edges contains $k_i$ edges of colour $i$ for
$i\in\{1,2,3\}$, then
$$
k_1\equiv k_2\equiv k_3\equiv k\pmod 2.
$$
\end{theorem}

Now let $M$ be a loopless cubic multipole that cannot be
properly $3$-edge-coloured. Then $M$ has a proper
$4$-edge-colouring with colours from the set
$\mathbf{C}=\mathbb{Z}_2\times\mathbb{Z}_2$. Such a colouring
will not be a $\mathbb{Z}_2\times\mathbb{Z}_2$-flow anymore
since every vertex incident with an edge coloured $0$ will have
a non-zero outflow. It is natural to require the colour $0$ to
be used as little as possible, that is, to require the set of
edges coloured $0$ to be the minimum-size colour class. Such a
$4$-edge-colouring will be called \textit{minimum}. In a
minimum $4$-edge-colouring of $M$ every edge $e$ coloured $0$
must be adjacent to edges of all three non-zero colours; in
particular, $e$ must be a proper edge. It follows that exactly
one colour around $e$ appears twice.

By summing the outflows at vertices incident with edges
coloured $0$ we obtain the following useful result due to
Fouquet~\cite[Theorem~1]{Fouq} and Steffen
\cite[Lemma~2.2]{Steffen:class}.

\begin{theorem}\label{thm:parity2}
Let $\phi$ be a minimum $4$-edge-colouring of a loopless cubic
multipole $M$ with $m$ edges coloured $0$, and for
$i\in\{1,2,3\}$ let $m_i$ denote the number of those edges
coloured $0$ that are adjacent to two edges coloured $i$. Then
$$m_1\equiv m_2\equiv m_3\equiv m\pmod{2}.$$
\end{theorem}

We finish the discussion of colourings with the definition of
the standard recolouring tool, a Kempe chain. Let $M$ be a
cubic multipole whose edges have been properly coloured with
colours from the set
$\{0,1,2,3\}=\mathbb{Z}_2\times\mathbb{Z}_2$. For any two
distinct colours $i,j\in\{1,2,3\}$ we define an
$i$-$j$-\textit{Kempe chain} $P$ to be a non-extendable walk
that alternates the edges with colours $i$ and $j$. Clearly,
$P$ is either an even circuit, or is a path that ends with
either a semiedge or with a vertex incident with an edge
coloured $0$. It is easy to see that switching the colours $i$
and $j$ on $P$ gives rise to a new proper $4$-edge-colouring of
$M$. Furthermore, if the original colouring was a minimum
$4$-edge-colouring, so is the new one.

\medskip
\noindent\textbf{2.4. Snarks.} A snark is, essentially, a
nontrivial cubic graph that has no $3$-edge-colouring. Precise
definitions vary depending on what is to be considered
``nontrivial''. In many papers, especially those dealing with
snark constructions, snarks are required to be cyclically
$4$-edge-connected and have girth at least $5$; see for example
\cite{C03, FMS-survey}. However, in~\cite{BrPS, hlund} the
girth requirement is dropped, demanding snarks to be cyclically
$4$-edge-connected but allowing them to have $4$-cycles.

Another group of papers, especially those dealing with the
structural analysis of snarks, adopts the widest possible
definition of a snark, permitting all kinds of trivial features
such as triangles, digons and even bridges; see, for
example~\cite{CCW, ChS, NS:dec}. In this paper, our usage of
the term snark agrees with the latter group: we define a
\textit{snark} to be a connected cubic graph that cannot be
$3$-edge-coloured.

This paper deals with snarks that are far from being
$3$-edge-colourable. Two measures of uncolourability will be
prominent in this paper.  The \emph{oddness} $\omega(G)$ of a
bridgeless cubic graph $G$ is the smallest number of odd
circuits in a 2-factor of $G$. The \emph{resistance} $\rho(G)$
of a cubic graph $G$ is the smallest number of edges of $G$
which have to be removed in order to obtain a colourable graph.
Obviously, if $G$ is colourable, then $\omega(G)=\rho(G)=0$. If
$G$ is uncolourable, then both $\omega(G)\ge 2$ and $\rho(G)\ge
2$. Furthermore, $\rho(G)\le\omega(G)$ for every bridgeless
cubic graph $G$.

The following lemma is due to Steffen~\cite{Steffen:class}.

\begin{lemma}\label{lemma:res2odd2}
Let $G$ be a bridgeless cubic graph. Then $\rho(G)=2$ if and
only if $\omega(G)=2$.
\end{lemma}

One of the methods of constructing snarks from smaller ones
uses I-extensions (cf. Section~2.2). The following result
from~\cite{NS:dec} tells us when an I-extension of a snark is
again a snark.

\begin{lemma}\label{lemma:I-ext}
Let $G$ be a snark and $e$ and $f$ be distinct edges of $G$.
Then $G{(e,f)}$ is a snark if and only if the graph $G-\{e,f\}$
is uncolourable.
\end{lemma}

Another method of constructing snarks is based on extending
multipoles to cubic graphs, see~\cite{ChS}. If the multipole in
question is uncolourable, it can be extended to a snark simply
by restoring $3$-regularity. We are therefore interested in
extending colourable multipoles. For $k\ge 2$, we say that a
$k$-pole $M$ \textit{extends} to a snark if there exists a
colourable multipole $N$ such that $M*N$ is a snark. The graph
$M*N$ is called a snark \textit{extension} of $M$.

Given a $k$-pole $M$ with semiedges $e_1,e_2,\dots,e_k$, we
define its \textit{colouring set} to be the following set of
$k$-tuples:
$$\Col(M)=\left\{\phi(e_1)\phi(e_2)\dots\phi(e_k);\,
\text{$\phi$ is a $3$-edge-colouring of $M$}\right\}.$$ Note
that the set $\Col(M)$ depends on the ordering in which the
semiedges are listed. We therefore implicitly assume that such
an ordering is given. As the colourings ``inside'' a multipole
can usually be ignored, we define two multipoles $M$ and $N$
to be \textit{colour-equivalent} if $\Col(M)=\Col(N)$.

Any colouring of a colourable multipole can be changed to a
different colouring by permuting the set of colours. The
particular colour of a semiedge is therefore not important, it
is only important whether it equals or differs from the colour
of any other semiedge. By saying this we actually define
the \textit{type} of a colouring $\phi$ of a multipole $M$: it
is the lexicographically smallest sequence of colours assigned
to the semiedges of $M$ which can be obtained from $\phi$ by
permuting the colours.

By the Parity Lemma, each colouring of a $4$-pole has one of
the following types: $1111$, $1122$, $1212$, and $1221$.
Observe that every colourable $4$-pole admits at least two
different types of colourings. Indeed, we can start with any
colouring and switch the colours along an arbitrary Kempe chain
to obtain a colouring of another type. Colourable $4$-poles
thus can have two, three, or four different types of
colourings. Those attaining exactly two types are particularly
important for the study of snarks; we call them
\textit{colour-open} $4$-poles, as opposed to
\textit{colour-closed} multipoles discussed in more detail
in~\cite{NS:dec}.

The following result appears in~\cite{ChS}.

\begin{proposition}\label{prop:ext}
A colourable $4$-pole extends to a snark if and only if it is
colour-open.
\end{proposition}

A $4$-pole $M$ will be called \textit{isochromatic} if its
semiedges can be partitioned into two pairs such that in every
colouring of $M$ the semiedges within each pair are coloured
with the same colour. A $4$-pole $M$ will be called
\textit{heterochromatic} of its semiedges can be partitioned
into two pairs such that in every colouring of $M$ the
semiedges within each pair are coloured with distinct colours.
The pairs of semiedges of an isochromatic or a heterochromatic
$4$-pole mentioned above will be called \textit{couples}.

Note that the $4$-pole $C_4$ obtained from a $4$-cycle by
attaching one dangling edge to every vertex is colour-closed,
and hence neither isochromatic nor heterochromatic. Indeed,
with respect to a cyclic ordering of its semiedges it admits
colourings of three types, namely $1111$, $1122$, and $1221$
(but not $1212$). In particular, if a snark $G$ contains a
$4$-cycle $C$, then, as is well-known, $G-V(C)$ stays
uncolourable.

The following two results are proved in~\cite{ChS}:

\begin{proposition}\label{prop:iso}
Every colour-open $4$-pole is either isochromatic or
heterochromatic, but not both. Moreover, it is isochromatic if
and only if it admits a colouring of type $1111$.
\end{proposition}

\begin{proposition}\label{prop:unique2}
Every colour-open $4$-pole can be extended to a snark by adding
at most two vertices, and such an extension is unique. A
heterochromatic multipole extends by joining the semiedges
within each couple, that is, by adding no new vertex. An
isochromatic multipole extends by attaching the semiedges of
each couple to a new vertex, and by connecting these two
vertices with a new edge.
\end{proposition}

Colour-open $4$-poles can be combined to form larger $4$-poles
from smaller ones by employing partial junctions: we take two
$4$-poles $M$ and $N$, choose two semiedges in each of them,
and perform the individual junctions. In general, such a
junction need not respect the structure of the couples of the
$4$-poles participating in the operation. In this manner it may
happen that, for example, a partial junction of two
heterochromatic $4$-poles results in an isochromatic dipole or
in a heterochromatic dipole. In Theorem~\ref{thm:4decomp}, one
of our decomposition theorems, partial junctions of $4$-poles
will occur in the reverse direction.

\section{Decomposition theorems}\label{sec:decomp}

The aim of this section is to show that every snark with
oddness at least $4$, cyclic connectivity~$4$, and minimum
number of vertices can be decomposed into two smaller
cyclically $4$-edge-connected snarks $G_1$ and $G_2$ by
removing a cycle-separating $4$-edge-cut, adding at most two
vertices to each of the components, and by restoring
$3$-regularity. This will be proved in two steps --
Theorem~\ref{thm:closure} and Theorem~\ref{thm:4decomp}.

Theorem~\ref{thm:closure} is a decomposition theorem for
cyclically $4$-edge-connected cubic graphs proved in 1988 by
Andersen et al.~\cite[Lemma~7]{AFJ}. Roughly speaking, it
states that every cubic graph $G$ whose cyclic connectivity
equals $4$ can be decomposed into two smaller cyclically
$4$-edge-connected cubic graphs $G_1$ and $G_2$ by removing a
cycle-separating $4$-edge-cut, adding two vertices to each of
the components, and by restoring $3$-regularity. Our proof is
different from the one in~\cite{AFJ} and provides useful
insights into the problem. For instance, it offers the
possibility to determine conditions under which it is feasible
to extend a $4$-pole to a cyclically $4$-edge-connected cubic graph
 by adding two isolated edges rather than by adding two
new vertices.

Theorem~\ref{thm:4decomp} deals with a particular situation
where the cyclically $4$-edge-connected cubic graph $G$ in
question is a snark. As explained in the previous section,
every snark containing a cycle-separating $4$-edge-cut that
leaves a colour-open component can be decomposed into two
smaller snarks $G_1$ and $G_2$ by removing the cut, adding
\emph{at most two vertices} to each of the components, and by
restoring $3$-regularity. Unfortunately, $G_1$ or $G_2$ are not
guaranteed to be cyclically $4$-edge-connected because snark
extensions forced by the colourings need not coincide with
those forced by the cyclic connectivity (see
Example~\ref{ex:forced_completion} below). Moreover,
Proposition~\ref{prop:unique2} suggests that restoring
$3$-regularity by adding no new vertices, that is, by joining
pairs of the four $2$-valent vertices to each other in one of
the components, may be necessary in order for $G_1$ or $G_2$ to
be a snark. If this is the case, Theorem~\ref{thm:closure}
cannot be applied. Nevertheless, Theorem~\ref{thm:4decomp}
shows that if $G$ is a smallest nontrivial snark with oddness
at least $4$, then we can form $G_1$ and $G_2$ in such a way
that they indeed will be cyclically $4$-edge-connected snarks.

\begin{example}\label{ex:forced_completion}
We give an example of a cyclically $4$-edge-connected snark in
which a decomposition along a given cycle-separating
$4$-edge-cut forces one of the resulting smaller snarks to have
cyclic connectivity smaller than $4$. To construct such a snark
take the Petersen graph and form a $4$-pole $H$ of order 10 by
severing two non-adjacent edges and a $4$-pole $I$ of order 8
by removing two adjacent vertices. It is easy to see that $H$
is heterochromatic with couples being formed by the semiedges
obtained from the same edge, and $I$ is isochromatic with
couples formed by the semiedges formerly incident with the same
vertex. Let us create a cubic graph $G$ by arranging two copies
of $H$ and one copy of $I$ into a cycle, and by performing
junctions that respect the structure of the couples. The
partial junction of two copies of $H$ contained in $G$, denoted
by $H^2$, is again a heterochromatic $4$-pole, so $G$ is a
junction of an isochromatic $4$-pole $I$ with a heterochromatic
$4$-pole, and therefore a snark. Furthermore, the cyclic
connectivity of $G$ equals $4$. Let us decompose $G$ by
removing from $G$ the $4$-edge-cut $S$ separating $I$ from
$H^2$ and by completing each of the components to a snark.
Proposition~\ref{prop:unique2} implies that $I$ can be
completed to a copy $G'$ of the Petersen graph while $H^2$
extends to a snark $G''$ of order $20$ by joining the semiedges
within each couple, that is, by adding no new vertex. The same
Proposition states that the decomposition of $G$ into $G'$ and
$G''$ is uniquely determined by $S$. However, $G''$ has a
cycle-separating $2$-edge-cut connecting the two copies of $H$
contained in it. Therefore the low connectivity of $G''$ is
unavoidable.
\end{example}

We proceed to Theorem~\ref{thm:closure}. It requires one
auxiliary result about comparable cuts. Let $S$ and $T$ be two
edge-cuts in a graph $G$. Let us denote the two components of
$G-S$ by $H_1$ and $H_2$ and those of $G-T$ by $K_1$ and $K_2$.
The cuts $S$ and $T$ are called \emph{comparable} if
$H_i\subseteq K_j$ or $K_j\subseteq H_i$ for some
$i,j\in\{1,2\}$.

\begin{lemma}\label{lemma:comparable2-cuts}
Let $G$ be a cyclically $4$-edge-connected cubic graph and let
$K$ be a component arising from the removal of a
cycle-separating $4$-edge-cut from $G$. Then any two nontrivial
$2$-edge-cuts in $K$ are comparable, or $K$ is a $4$-cycle.
\end{lemma}

\begin{proof}
Let $S$ be the cycle-separating $4$-edge-cut that separates $K$
from the rest of $G$, and let $A=\{a_1,a_2,a_3,a_4\}$ be the
set of the vertices of $K$ incident with an edge from $S$.
Since $S$ is independent, the vertices of $A$ are pairwise
distinct. Proposition~\ref{prop:cc-properties}~(i) implies that
$K$ is $2$-connected. It follows that for every nontrivial
$2$-edge-cut $Q$ in $K$ the graph $K-Q$ consists of two
components, each containing exactly two vertices of $A$.

Let $R$ and $T$ be two nontrivial 2-edge-cuts in $K$. Denote
the components of $K-R$ by $X_1$ and $X_2$, and those of $K-T$
by $Y_1$ and $Y_2$. Observe that the subgraphs $X_i\cap Y_j$
for $i,j\in\{1,2\}$ need not all be non-empty. Let $a$ be the
number of edges between $X_1\cap Y_1$ and $X_1\cap Y_2$,
$b$~the number of edges between $X_1\cap Y_1$ and $X_2\cap
Y_2$, $c$~the number of edges between $X_1\cap Y_1$ and
$X_2\cap Y_1$, $d$ the number of edges between $X_1\cap Y_2$
and $X_2\cap Y_1$, $e$ the number of edges between $X_2\cap
Y_1$ to $X_2\cap Y_2$, and finally $f$ the number of edges
between $X_1\cap Y_2$ and $X_2\cap Y_2$; see
Figure~\ref{fig:sets}.


\begin{figure}[htbp]
  \centerline{
     \scalebox{0.5}{
       \input{figures/sets.tex}
     }
  }
\caption{Crossing edge-cuts $R$ and $T$.}
\label{fig:sets}
\end{figure}
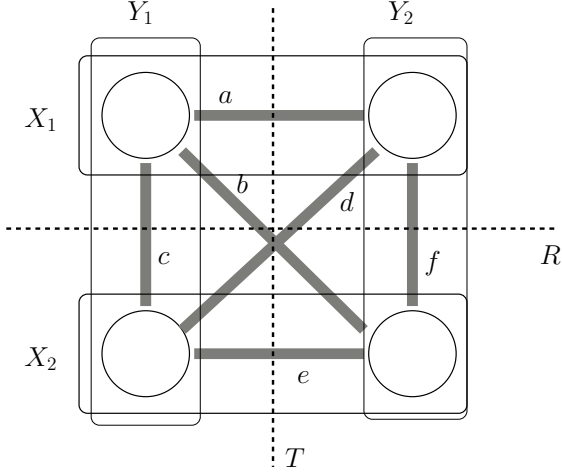

If at least one of the sets $X_1\cap Y_2$, $X_2\cap Y_1$, and
$X_1\cap Y_1$ is empty, then the definition of comparable cuts
directly implies that the cuts $R$ and $T$ are comparable, as
required. Thus we can assume that all the subgraphs $X_i\cap
Y_j$ are nonempty. Our aim is to show that in this case $K$ is
a $4$-cycle. We start by showing that each of the subgraphs
$X_1\cap Y_1$, $X_1\cap Y_2$, $X_2\cap Y_1$, and $X_2\cap Y_2$
contains exactly one element of $A$. Suppose that one of them,
say $X_1\cap Y_1$, contains no vertex from $A$. Since both $R$
and $T$ separate the vertices from $A$ in such a way that both
components contain two vertices from $A$, we deduce that both
$X_1\cap Y_2$ and $X_2\cap Y_1$ contain two vertices from $A$
each, while $X_2\cap Y_2$ contains no vertex from $A$. Now
$|\delta_K(X_1\cap Y_1)|=a+b+c\ge 3$, because $G$ has no
bridges and no 2-edge-cuts. Further, since $X_1\cap Y_2$
contains two vertices from $A$ and $G$ is cyclically
$4$-edge-connected, we see that $|\delta_K(X_1\cap
Y_2)|=a+d+f\ge 2$. However, $R$ is a 2-cut, so $b+c+d+f=2$.
Therefore $2a\ge 3$ and hence $a\ge 2$. Similarly, $e\ge 2$.
But then $|T|\ge a+e\ge 4$, which contradicts the fact that $T$
is a 2-edge-cut. Thus all the subgraphs $X_i\cap Y_j$ contain
an element of $A$, which in turn implies that each $X_i\cap
X_j$ contains exactly one vertex from~$A$.

To finish the proof we show that $a=c=e=f=1$ and $b=d=0$.
Suppose that $a=2$. Since $T$ is a 2-edge-cut, we have that
$b=d=e=0$. Now $c+d+e\ge 2$ and $b+e+f\ge 2$ because $G$ is
$3$-edge-connected, so $c\ge2$ and $f\ge2$, and hence $|R|\ge
c+f\ge 4$, a contradiction. Thus $a\le 1$. Similarly, we can
derive that $c\le 1$, $e\le 1$, and $f\le 1$. If $b=2$, then
$a=c=d=e=f=0$ implying that $G$ has a bridge, which is absurd.
Hence $b\le 1$ and similarly $d\le 1$. Suppose that $a=0$. As
$G$ is $3$-edge-connected, we have $2\le a+b+c=b+c\le 2$ and
similarly $2\le a+d+f=d+f\le 2$. It follows that that
$b=c=d=f=1$ and hence $|R|= b+c+d+f=4$, which contradicts the
fact that $R$ is a $2$-cut. Therefore $a=1$ and similarly
$c=e=f=1$, which also implies that $b=d=0$. Finally, every
subgraph $X_i\cap Y_j$ has $|\delta_G(X_i\cap Y_j)|=3$, so
$X_i\cap Y_j$ is acyclic and therefore, by
Lemma~\ref{lemma:pvmultipol}, a single vertex. In other words,
$K$ is a $4$-cycle. This completes the proof.
\end{proof}

We are ready to prove the decomposition theorem of Andresen et
al.~\cite{AFJ}.

\begin{theorem}\label{thm:closure}
Let $G$ be a cyclically $4$-edge-connected cubic graph with a
cycle-separating $4$-edge-cut whose removal leaves components
$G_1$ and $G_2$. Then each of $G_1$ and $G_2$ can be extended
to a cyclically $4$-edge-connected cubic graph by adding two
adjacent vertices and restoring $3$-regularity.
\end{theorem}

\begin{proof}
It suffices to prove the statement for $G_1$. If $G_1$ is a
$4$-cycle, we can easily extend it to the complete bipartite
graph $K_{3,3}$ which is cyclically $4$-edge-connected, as
required. We therefore assume that $G_1$ is not a $4$-cycle.
Let $A=\{a_1, a_2,a_3, a_4\}$ be the set of vertices of $G_1$
incident with an edge of $\delta_G(G_1)$. By
Lemma~\ref{lemma:comparable2-cuts}, every 2-edge-cut in $G_1$
separates the vertices of $A$ into the same two 2-element sets,
say $\{a_1, a_2\}$ from $\{a_3, a_4\}$. We extend $G_1$ to a
cyclically 4-edge-connected cubic graph $\tilde G_1$ as
follows. Let us take two new vertices $x_1$ and $x_2$ and
construct $\tilde G_1$ from $G_1$ by adding to $G_1$ the edges
$x_1x_2$, $x_1a_1$, $x_1a_3$, $x_2a_2$, and $x_2a_4$, see
Figure~\ref{fig:close}. We now verify that $\tilde G_1$ is
indeed cyclically 4-edge-connected.

\begin{figure}[htbp]
  \centerline{
     \scalebox{0.4}{
       \input{figures/closure2}
     }
  }
\caption{Extending $G_1$ to $\tilde G_1$.}
\label{fig:close}
\end{figure}
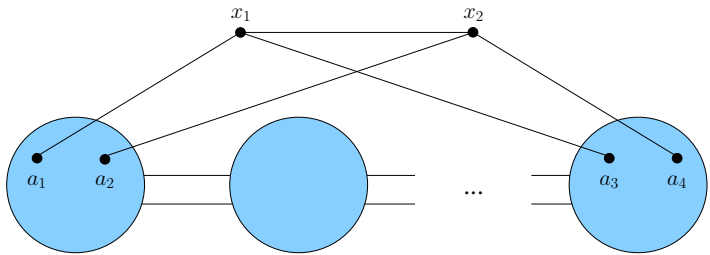

Suppose to the contrary that $\tilde G_1$ is not cyclically
4-edge-connected. Then $\tilde G_1$ has a minimum-size
cycle-separating edge-cut $F$ such that $|F|<4$. Let $H_1$ and
$H_2$ be the components of $G_1-F$. The cut $F$ cannot consist
entirely of edges of $G_1\cup \delta_G(G_1)$ for otherwise $F$
would be a cycle-separating edge-cut of $G$ of size smaller
than $4$. Therefore the edge $x_1x_2$ is contained in $F$.
Since $F$ is an independent cut, the edges $x_1a_1$, $x_1a_3$,
$x_2a_2$, and $x_2a_4$ do not belong to $F$. This in turn
implies that $a_1$ and $a_3$ belong to one component of $\tilde
G_1-F$ while $a_2$ and $a_4$ belong to the other component of
$\tilde G_1-F$; without loss of generality, let $a_1$ and $a_3$
belong to $H_1$. Since $G_1$ contains no bridge, there exist
edges $e_1$ and $e_2$ in $G_1$ such that
$F=\{x_1x_2,e_1,e_2\}$. But then $\{e_1,e_2\}$ is a 2-edge cut
in $G_1$ that separates the set $\{a_1,a_3\}$ from
$\{a_2,a_4\}$, which is a contradiction. This completes the
proof.
\end{proof}

Before proving the second main result of this section we need
the following fact.

\begin{proposition} \label{prop:odd2}
Let $G$ be a cubic graph with a cycle-separating $4$-edge-cut
whose removal leaves components $G_1$ and $G_2$. If both $G_1$
and $G_2$ are $3$-edge-colourable, then $\omega(G)\le 2$.
\end{proposition}

\begin{proof}
Assume that both $G_1$ and $G_2$ are $3$-edge-colourable. If
$G$ is $3$-edge-colourable, then $\omega(G)=0$. Therefore we
may assume that $G$ is not $3$-edge-colourable. In this
situation $G_1$ admits two types of colourings and $G_2$ admits
the other two types of colourings. One of them, say $G_1$ has a
colouring $\phi_1$ of the type $1111$; by
Proposition~\ref{prop:iso}, $G_1$ is isochromatic and $G_2$ is
heterochromatic. Parity Lemma implies that if we take an
arbitrary $3$-edge-colouring $\phi_2$ of $G_2$, then exactly
two colours occur on the dangling edges of $G_2$. Let $e$ and
$f$ be any two of the dangling edges that receive the same
colour. Then, after permuting the colours in $G_1$, if
necessary, $\phi_1$ and  $\phi_2$ can be easily combined to a
$3$-edge-colouring of $G-\{e,f\}$. This shows that $\rho(G)=2$
and therefore $\omega(G)=2$.
\end{proof}

Now we are in position to prove our second decomposition
theorem.

\begin{theorem} \label{thm:4decomp}
Let $G$ be a snark with oddness at least $4$, cyclic
connectivity $4$, and minimum number of vertices. Then $G$
contains a cycle-separating $4$-edge-cut $S$ such that both
components of $G-S$ can be extended to a cyclically
$4$-edge-connected snark by adding at most two vertices.
\end{theorem}

In fact, we prove the following stronger and more detailed
result which will also be needed in our next paper \cite{GMS}.

\begin{theorem} \label{thm:4decomp-detailed}
Let $G$ be a snark with oddness at least $4$, cyclic
connectivity $4$, and minimum number of vertices. Let $S$ be a
cycle-separating $4$-edge-cut in $G$ whose removal leaves
components $G_1$ and $G_2$. Then, up to permutation of the
index set $\{1,2\}$, exactly one of the following occurs.
\begin{itemize}
\item[{\rm (i)}] Both $G_1$ and $G_2$ are uncolourable, in
    which case each of them can be extended to a
    cyclically $4$-edge-connected snark by adding two
    vertices.
\item[{\rm (ii)}]  $G_1$ is uncolourable and $G_2$ is
    heterochromatic, in which case $G_1$ can be extended
    to a cyclically $4$-edge-connected snark by adding two
    vertices, and $G_2$ can be extended to a cyclically
    $4$-edge-connected snark by adding two isolated edges.
\item[{\rm (iii)}]$G_1$ is uncolourable and $G_2$ is
    isochromatic, in which case $G_1$ can be extended to a
    cyclically $4$-edge-connected snark by adding two
    vertices, and $G_2$ can be extended to a
    cyclically $4$-edge-connected snark by adding two
    vertices, except possibly
    $\zeta(G_2)=2$. In the latter case, $G_2$ is a partial
    junction of two colour-open $4$-poles, which may be
    isochromatic or heterochromatic in any combination.
\end{itemize}
\end{theorem}

\begin{proof}
Let $G$ be a snark  with $\omega(G)\ge 4$,  $\zeta(G)=4$, and
with minimum number of vertices. Let $S=\{s_1,s_2,s_3,s_4\}$ be
an arbitrary fixed cycle-separating 4-edge-cut in $G$, and let
$G_1$ and $G_2$ be the components of $G-S$. According to
Proposition~\ref{prop:odd2}, at least one of $G_1$ and $G_2$ is
uncolourable. If both $G_1$ and $G_2$ are uncolourable, we can
extend each of them to a cyclically $4$-edge-connected snark by
applying Theorem~\ref{thm:closure}, establishing (i). For the
rest of the proof we may therefore assume that $G_2$ is
colourable and $G_1$ is not. Again, $G_1$ can be extended to a
cyclically $4$-edge-connected snark by
Theorem~\ref{thm:closure}. Let $\tilde G_1$ be an extension of
$G_1$ to a cyclically $4$-edge-connected snark by adding two
adjacent vertices $y_1$ and $y_2$ according to
Theorem~\ref{thm:closure}. Without loss of generality we may
assume that the vertex $y_1$ is incident with the edges $s_1$
and $s_2$ while $y_2$ is incident with $s_3$ and $s_4$.

As regards $G_2$, we prove that either (ii) or (iii) holds. Our
first step in this direction is showing that $G_2$ can be
extended to a snark. In view of Proposition~\ref{prop:ext}, this
amounts to verifying that $G_2$ is colour-open.

\medskip\noindent
Claim 1. \textit{The $4$-pole $G_2$ is colour-open.}

\medskip\noindent
Proof of Claim~1. Suppose to the contrary that $G_2$ is not
colour-open. This means that it has at least three types of
colourings. Since $G$ is a smallest cyclically
$4$-edge-connected snark with oddness at least $4$ and $\tilde
G_1$ is a cyclically $4$-edge-connected snark with fewer
vertices than $G$, we infer that $\omega(\tilde G_1)=2$. By
Lemma~\ref{lemma:res2odd2}, there exist two nonadjacent edges
$e_1$ and $e_2$ in $\tilde{G}_1$ such that $\tilde
G_1-\{e_1,e_2\}$ is colourable.  Equivalently, by
Lemma~\ref{lemma:I-ext}, the cubic graph
$\tilde{G}_1{(e_1,e_2)}$ is colourable.

We claim that the edge $y_1y_2$ is one of $e_1$ and $e_2$.
Suppose not. Then both $e_1$ and $e_2$ have at least one
end-vertex in~$G_1$. As mentioned, $\tilde{G}_1{(e_1,e_2)}$ is
a colourable cubic graph. Hence $G_1{(e_1,e_2)}$ is a
colourable $4$-pole, and therefore it has at least two types of
colourings. Since $G_2$ has at least three of the four types,
both $G_1{(e_1,e_2)}$ and $G_2$ admit colourings of the same
type. These colourings can be combined into a colouring of
$G{(e_1,e_2)}$, implying that $G-\{e_1,e_2\}$ is also
colourable. However, from Lemma~\ref{lemma:res2odd2} we get
that $\omega(G)=2$, which is a contradiction proving that one
of $e_1$ and $e_2$ coincides with $y_1y_2$.

Assuming that $y_1y_2=e_1$, let us consider a minimum
$4$-edge-colouring $\phi_1$ of $\tilde{G}_1$ where $e_1$ and
$e_2$ are the only edges of $\tilde{G}_1$ coloured $0$.
Theorem~\ref{thm:parity2} implies that there exist a
unique non-zero colour that is repeated at both $e_1$ and
$e_2$. Without loss of generality we may assume that the
repeated colour is $1$ and that $\phi_1(s_1)=\phi_1(s_3)=1$,
$\phi_1(s_2)=2$, and $\phi_1(s_4)=3$. In this situation, $G_2$
cannot have a colouring of type $1212$ for otherwise we could
combine this colouring with $\phi_1$ to produce a
$3$-edge-colouring of $G-\{e_2,s_4\}$, which is impossible
since $\omega(G)\ge 4$. Therefore $G_2$ has colourings of all
the remaining three types $1111$, $1122$, and $1221$.

Consider the $1$-$2$-Kempe chain $P$ in $\tilde G_1$ with
respect to the colouring $\phi_1$ beginning at the vertex
$y_2$. Clearly, the other end of $P$ must be the end-vertex of
$e_2$ incident with edges of colours $1$, $3$, and $0$. If $P$
does not pass through the vertex $y_1$, we switch the colours
on $P$ producing a $4$-edge-colouring $\phi_1'$ of $\tilde G_1$
where $\phi_1'(s_1)=1$, $\phi_1'(s_2)=\phi_1'(s_3)=2$, and
$\phi_1'(s_4)=3$. However, $\phi_1'$ can be combined with a
colouring of $G_2$ of type $1221$ to obtain a
$3$-edge-colouring of $G-\{e_2,s_4\}$, which is impossible
since $\omega(G)\ge 4$. If $P$ passes through $y_1$, we switch
the colours only on the segment $P_0$ between $y_2$ and $y_1$,
producing an improper colouring $\phi_1''$ of $\tilde G_1$ with
$y_1$ being its only faulty vertex. Depending on whether $P_0$
ends with an edge coloured $1$ or $2$ we get
$\phi_1''(s_1)=\phi_1''(s_2)=\phi_1''(s_3)=2$ and
$\phi_1''(s_4)=3$, or $\phi_1''(s_1)=\phi_1''(s_2)=1$,
$\phi_1''(s_3)=2$, and $\phi_1''(s_4)=3$. In the latter case we
can combine $\phi_1''$ with a colouring of $G_2$ of type
$1122$, producing a $3$-edge-colouring of $G-\{e_2,s_4\}$. In
the former case we first interchange the colours $1$ and $2$ on
$G_1$ and then combine the resulting colouring with a colouring
of $G_2$ of type $1111$, again producing a $3$-edge-colouring
of $G-\{e_2,s_4\}$. Since $\omega(G)\ge 4$, in both cases we
have reached a contradiction. This establishes Claim~1.

\medskip

Proposition~\ref{prop:unique2} now implies that $G_2$ can be
extended to a snark $\bar G_2$ by adding at most two vertices.
Recall that such an extension is unique up to isomorphism and
depends only on whether $G_2$ is isochromatic or
heterochromatic. We discuss these two cases separately.

\medskip\noindent
\textbf{Case 1.} \textit{$G_2$ is isochromatic.} First note
that in this case $\bar G_2$ arises from $G_2$ by adding two
new vertices $x_1$ and $x_2$ joined by an edge and by attaching
each of the new vertices to the semiedges in the same couple.
From Proposition~\ref{prop:cc-properties}~(i) we get that
$\zeta(G_2)\ge 2$. If $\zeta(G_2)\ge 4$, then the same
obviously holds for $\bar{G_2}$. Assume that $\zeta(G_2)=3$,
and let $A$ denote the set of end-vertices in $G_2$ of the
edges of the edge-cut $S$. Note that $|A|=4$ because $S$ is
independent. Since $\zeta(G)=4$, every cycle separating
$3$-edge-cut $R$ in $G_2$ has the property that each component
of $G_2-R$ contains at least one vertex from $A$. This readily
implies that $\zeta(\bar G_2)\ge4$ and establishes the
statement (iii) whenever $\zeta(G_2)\ge 3$. It remains to
consider the case where $\zeta(G_2)=2$. Let $U$ be a
cycle-separating 2-edge-cut in $G_2$ and let $Q_1$ and $Q_2$ be
the components of $G_2-U$. Since $G$ is cyclically
$4$-edge-connected, each $Q_i$ contains exactly two vertices
from $A$ and thus both $Q_1$ and $Q_2$ are 4-poles. Each $Q_i$
is colourable because any $3$-edge-colouring of $G_2$ provides
one for $Q_i$. Furthermore, each $Q_i$ is colour-open, because
$G_2$ and hence also $Q_i$ has an extension to $\bar G_2$. Thus
$G_2$ is a partial junction of two colour-open $4$-poles. It is
not difficult to show that an isochromatic $4$-pole can arise
from a partial junction of any combination of isochromatic and
heterochromatic $4$-poles, as claimed.

\medskip

\medskip\noindent
\textbf{Case 2.} \textit{$G_2$ is heterochromatic.} In this
case $G_2$ arises from a snark by severing two independent
edges. Suppose to the contrary that $\bar G_2$ is not
cyclically $4$-edge-connected. Then $G_2$ has at least twelve
vertices, because there is only one $2$-edge-connected snark of
order less than twelve -- the Petersen graph  -- and its cyclic
connectivity equals $5$. Let us take a heterochromatic $4$-pole
$H$ of order $10$ obtained from the Petersen graph and
substitute $G_2$ in $G$ with $H$, creating a new cubic
graph~$G'$. Clearly, $G'$ is a snark of order smaller than $G$.
To derive a final contradiction with the minimality of $G$ we
show that $G'$ is cyclically 4-edge-connected and has oddness
at least~$4$.

\medskip\noindent
Claim 2. \textit{$\omega(G')\ge 4$}.

\medskip\noindent
Proof of Claim~2. Suppose to the contrary that $\omega(G')<4$.
Since $G_1$ is uncolourable and $G_1\subseteq G'$, we infer
that $\omega(G')=2$ which in turn implies that $\rho(G')=2$.
Therefore there exist edges $e_1$ and $e_2$ in $G'$ such that
$G'-\{e_1,e_2\}$ is colourable. In other words, $G'$ has a
minimum $4$-edge-colouring $\psi$ where $e_1$ and $e_2$ are the
only edges of $G'$ coloured $0$.

Since $G_1$ is uncolourable, at least one of $e_1$ and $e_2$
must have both end-vertices in~$G_1$. Without loss of
generality assume that at $e_1$ has both end-vertices in $G_1$.
If $e_2$ had at least one end-vertex in $G_1$, we could take a
$3$-edge-colouring of $G'-\{e_1,e_2\}$, remove $H$ and
reinstate $G_2$ coloured in such a way that the edges in
$S-\{e_2\}$ receive the same colours from $G_2$ as they did
from $H$; this is possible since $G_2$ and $H$ are
colour-equivalent. However, in this way we would produce
a $3$-edge-colouring of $G-\{e_1,e_2\}$, contrary to the
assumption that $\omega(G)=4$. Therefore $e_2$ has both ends in
$H$.

Since $H$ is heterochromatic, the edges of $S$ can be
partitioned into couples such that for every $3$-edge-colouring
of $H$ the colours of both edges within a couple are always
different. Let $\{s_i,s_j\}$ and $\{s_k,s_l\}$ be the couples
of $H$. Further, since $\psi$ is a minimum $4$-edge-colouring
of $G'$, all three non-zero colours are present on the edges
adjacent to each of $e_i$, one of the colours being
represented twice. By Theorem~\ref{thm:parity2}, the same colour
occurs twice at both $e_1$ and $e_2$, say colour $1$. If we
regard $\psi$ as a $\mathbb{Z}_2\times\mathbb{Z}_2$-valuation
and sum the outflows from vertices of $G_1$ we see that the
flow through $S$ equals
$\psi(s_1)+\psi(s_2)+\psi(s_3)+\psi(s_4)=1$. Hence, the
distribution of colours in the couples of $S$, the set
$\{\{\psi(s_i),\psi(s_j)\},\{\psi(s_k),\psi(s_l)\}\}$, must
have one of the following four forms:
\begin{itemize}
\item[]
$D_1=\{\{1,1\},\{2,3\}\}$

\item[]
$D_2=\{\{1,2\},\{1,3\}\}$

\item[]
$D_3=\{\{2,2\},\{2,3\}\}$

\item[]
$D_4=\{\{2,3\},\{3,3\}\}$.
\end{itemize}

We now concentrate on the restriction of $\psi$ to $G_1$ and
show that it can be modified to a $4$-edge-colouring $\lambda$
of $G_1$ with distribution either $D_2$ or $D_3$. If the
colouring $\psi$ of $G'$ has distribution $D_4$, we can simply
interchange the colours $2$ and $3$ to obtain the
distri\-bution~$D_3$. Assume that $\psi$ has distribution
$D_1$. Let us consider the unique end-vertex $u$ of $e_1$ in
$G_1$ such that the edges incident with $u$ receive colours
$1$, $3$, and~$0$ from $\psi$. The $1$-$2$-Kempe chain $P$
starting at $u$ ends with a vertex incident with $e_2$, which
means that $P$ traverses~$S$. Let $s$ be the first edge of $S$
that belongs to $P$. If $\psi(s)=1$, then the desired
$4$-edge-colouring $\lambda$ of $G_1$ with distribution $D_2$
can be obtained by the Kempe switch on the segment of $P$ that
ends with $s$ and by a subsequent permutation of colours
interchanging $1$ and $2$. If $\psi(s)=2$, then a
$4$-edge-colouring of $G_1$ with distribution $D_3$ can be
obtained similarly. In both cases, $e_1$ is the only edge
coloured $0$ under $\lambda$.

If $\lambda$ has distribution $D_2$, then $\lambda$ and a
$3$-edge-colouring of $H$ of type $1212$ can be combined to a
$3$-edge-colouring of $G'-\{e_1,s_l\}$. However, as observed
earlier, by removing $H$ and reinstating $G_2$ we could produce
a $3$-edge-colouring of $G-\{e_1,s_l\}$, which is impossible
because $\omega(G)\ge 4$. If $\lambda$ has distribution $D_3$,
we can similarly combine $\lambda$ with a $3$-edge-colouring of
$H$ of type $1221$ to a $3$-edge-colouring of $G'-\{e_1,s_i\}$
which is impossible for the same reason. This contradiction
completes the proof of Claim~2.

\medskip\noindent
Claim 3. \textit{ $\zeta(G')=4$.}

\medskip\noindent
Proof of Claim~3. Suppose to the contrary that $\zeta(G')<4$.
Let $S'$ be a minimum size cycle-separating edge-cut in $G'$.
If all the edges of $S'$ had at least one end vertex in $G_1$,
then $S'$ would be a cycle-separating cut also in $G$, which is
impossible. Therefore at least one edge of $S'$ has both ends
in $H$, which means that $S'$ intersects $H$. Since $H$ is
connected, we conclude that $S'_H=S'\cap E(H)$ is an edge-cut
of $H$. Note that $S'_H$ is an independent set of edges, so
$S'_H$ must be a cycle-separating edge-cut in $H$. Recall,
however, that $H$ arises from the Petersen graph by severing
two independent edges $e$ and~$f$. It follows that
$S'_H\cup\{e,f\}$ is a cycle-separating edge-cut in the
Petersen graph. Hence, $|S'_H\cup\{e,f\}|\ge 5$, and
consequently $3\le |S'_H|\le |S'|\le 3$. This shows that
$S'_H=S'$ and therefore $S'$ is completely contained in $H$; in
particular $S'\cap S=\emptyset$. Because $S'$ is an edge-cut of
the entire $G'$, all the edges of $S$ must join $G_1$ to the
same component of $H-S'$. On the other hand, the Petersen graph
is cyclically $5$-edge-connected, therefore both $e$ and $f$
have end-vertices in different components of $H-S'$. The way
how $G'$ was constructed from $G$ now implies that the set of
end-vertices of $S$ in $H$ coincides with the set of
end-vertices of $e$ and $f$. Therefore $S$ has an end-vertex in
each component of $H-S'$, contradicting the previous
observation. This contradiction establishes Claim~3 and
concludes the entire proof.
\end{proof}

We proceed to proving our second decomposition theorem.

\medskip\noindent
\textit{Proof of Theorem~\ref{thm:4decomp}.} Let $G$ be a snark
with oddness at least $4$, cyclic connectivity~$4$, and minimum
number of vertices. If $G$ contains a cycle-separating
$4$-edge-cut whose removal leaves either two uncolourable
components or one uncolourable component and one
heterochromatic component, then the conclusion follows directly
from Theorem~\ref{thm:4decomp-detailed}~(i) or (ii),
respectively. Otherwise one of the components is uncolourable
and the other one, denoted by $G_2$, is isochromatic. In this
case, $G_2$ contains a subgraph $K$ which is an atom, possibly
$K=G_2$. Clearly, $K$ is colourable and $\delta_G(K)$ is a
cycle-separating $4$-edge-cut. If $K$ is heterochromatic, then
the conclusion again follows from
Theorem~\ref{thm:4decomp-detailed}~(ii). Therefore we may
assume that $K$ is isochromatic. Since $4=\zeta(G)<5\le g(G)$,
we see that $K$ is a nontrivial atom and from
Proposition~\ref{prop:cc-properties}~(ii) we infer that
$\zeta(K)\ge 3$. Applying statement~(iii) of
Theorem~\ref{thm:4decomp-detailed} with $S=\delta_G(K)$ we
finally get the desired result. \hfill $\qed$

\newpage

\section{Main result}\label{sec:MainResult}

We are now ready to prove our main result.

\mainthm*

\begin{proof}
Let $G$ be a snark with oddness at least $4$, cyclic
connectivity $4$, and minimum order. We first prove that $G$
has girth at least $5$. By Proposition~\ref{prop:girth}, the
girth of $G$ is at least $4$. Suppose to the contrary that $G$
contains a $4$-cycle $C$, and let $S$ be the edge-cut
separating $C$ from the rest of $G$. Since $S$ is
cycle-separating, it has to satisfy one of the statements
(i)--(iii) of Theorem~\ref{thm:4decomp-detailed}. In the
notation of Theorem~\ref{thm:4decomp-detailed}, $C$ necessarily
plays the role of $G_2$, because it is colourable. In
particular,  $S$ does not satisfy (i). However, $S$ satisfies
neither (ii) because $G_2$ is not heterochromatic, nor (iii)
since $G_2$ is not isochromatic. Thus we have reached a
contradiction proving that the girth of $G$ is at least~$5$.

In Figure~\ref{fig:LMMS44} we have displayed a snark with
oddness at least $4$, cyclic connectivity $4$ on $44$ vertices.
It remains to show that there are no snarks of oddness at least
$4$ and cyclic connectivity $4$ with fewer than $44$ vertices.

Our main tool is Theorem~\ref{thm:4decomp}. It implies that
every snark with oddness at least $4$, cyclic connectivity~$4$,
and minimum number of vertices can be obtained from two smaller
cyclically $4$-edge-connected snarks $G_1$ and $G_2$ by the
following process:
\begin{itemize}
\item Form a $4$-pole $H_i$ from each $G_i$ by either
    removing two adjacent vertices or two nonadjacent edges
    and by retaining the dangling edges.
\item Construct a cubic graph $G$ by identifying the
    dangling edges of $H_1$ with those of $H_2$ after
    possibly applying a permutation to the dangling edges
    of $H_1$ or $H_2$.
\end{itemize}
Any graph $G$ obtained in this manner will be called a
\textit{$4$-join} of $G_1$ and $G_2$. Note that the well-known
operation of a \textit{dot product} of snarks~\cite{AVT,
Isaacs} is a special case of a $4$-join.

We proceed to proving that every snark with cyclic connectivity
$4$ on at most $42$ vertices has oddness $2$. If $G$ is a snark
with cyclic connectivity $4$ on at most $42$ vertices, then by
Theorem~\ref{thm:4decomp} it contains a cycle-separating
$4$-edge-cut $S$ such that both components $K_1$ and $K_2$ of
$G-S$ can be extended to snarks $G_1$ and $G_2$, respectively,
by adding at most two vertices; in other words, $G$ is a
$4$-join of $G_1$ and $G_2$. Clearly,
$|V(K_1)|+|V(K_2)|=|V(G)|\le 42$. Assuming that $|V(K_1)|\le
|V(K_2)|$ we see that $|V(K_1)|\ge 8$, because the smallest
cyclically $4$-edge-connected snark has $10$ vertices, and
hence $|V(K_2)|\le 34$. Therefore both $G_1$ and $G_2$ have
order at least 10 and at most 36.

Let $\mathcal{S}_{n}$ denote the set of all pairwise
non-isomorphic cyclically $4$-edge-connected snarks of order
not exceeding $n$. To finish the proof it remains to show that
every $4$-join of two snarks from $\mathcal{S}_{36}$ with at
most 42 vertices has oddness $2$. Unfortunately, verification
of this statement in a purely theoretical way is far beyond
currently available methods. The final step of our proof has
been therefore performed by a computer.

We have written a program which applies a $4$-join in all
possible ways to two given input graphs and have applied this
program to the complete list of snarks from the set
$\mathcal{S}_{36}$. More specifically, given an arbitrary pair
of input graphs, the program removes in all possible ways
either two adjacent vertices or two nonadjacent edges from each
of the graphs (retaining the dangling edges) and then
identifies the dangling edges from the first graph in the pair
with the dangling edges of the second graph, again in all
possible (i.e., $4! = 24$) ways. We also use the \textit{nauty}
library~\cite{nauty-website, mckay_14} to determine the orbits
of edges and edge pairs in the input graphs, so the program
only removes two adjacent vertices or two nonadjacent edges
once from every orbit of edges or edge pairs, respectively. The
resulting graphs can still contain isomorphic copies, therefore
we also use \textit{nauty} to compute a canonical labelling of
the graphs and remove the isomorphic copies.

Until now, only the set $\mathcal{S}_{34}$
has been known; it was determined by Brinkmann, H\"agglund,
Markstr\"om, and the first author~\cite{BrGHM} in 2013 and was
shown to contain exactly $27~205~766$ snarks. Using the program
\textit{snarkhunter}~\cite{BrGM, BrGHM} we have been able to
generate all cyclically 4-edge-connected snarks on 36 vertices,
thereby completing the determination of $\mathcal{S}_{36}$.
This took about 80 CPU years and yielded exactly $404~899~916$
such graphs. The size of $\mathcal{S}_{36}$ thus totals to
$432~105~682$ graphs. (The new list of snarks can  be
downloaded from the \textit{House of Graphs}~\cite{BCGM} at
\url{http://hog.grinvin.org/Snarks}).

Finally, we have performed all possible $4$-joins of two snarks
from $\mathcal{S}_{36}$ that produce a snark with  at most $42$
vertices and checked their oddness. This computation required
approximately 75 CPU days. We have used two independent
programs to compute the oddness of the resulting graphs (the
source code of these programs can be obtained
from~\cite{oddness-site}) and in each case the results of both
programs were in complete agreement. No snark of oddness
greater than $2$ among them was found, which completes the
proof of Theorem~\ref{thm:main}.
\end{proof}

\section{Remarks and open problems}\label{sec:remarks}

We have applied the $4$-join operation to all valid pairs of
snarks from $\mathcal{S}_{36}$ to construct cyclically
$4$-edge-connected snarks on $44$ vertices and checked their
oddness. In this manner we have produced $31$ cyclically
$4$-edge-connected snarks of oddness $4$, including the one
from Figure~\ref{fig:LMMS44}, all of them having girth $5$. The
most symmetric of them is shown in Figure~\ref{fig:sym44}. We
will describe and analyse these 31 snarks in the sequel of this
paper~\cite{GMS}, where we also prove that they constitute a
complete list of all snarks with oddness at least $4$, cyclic
connectivity $4$, and minimum number of vertices.

\begin{figure}[htbp]
	\centering
	\includegraphics[width=0.4\textwidth]{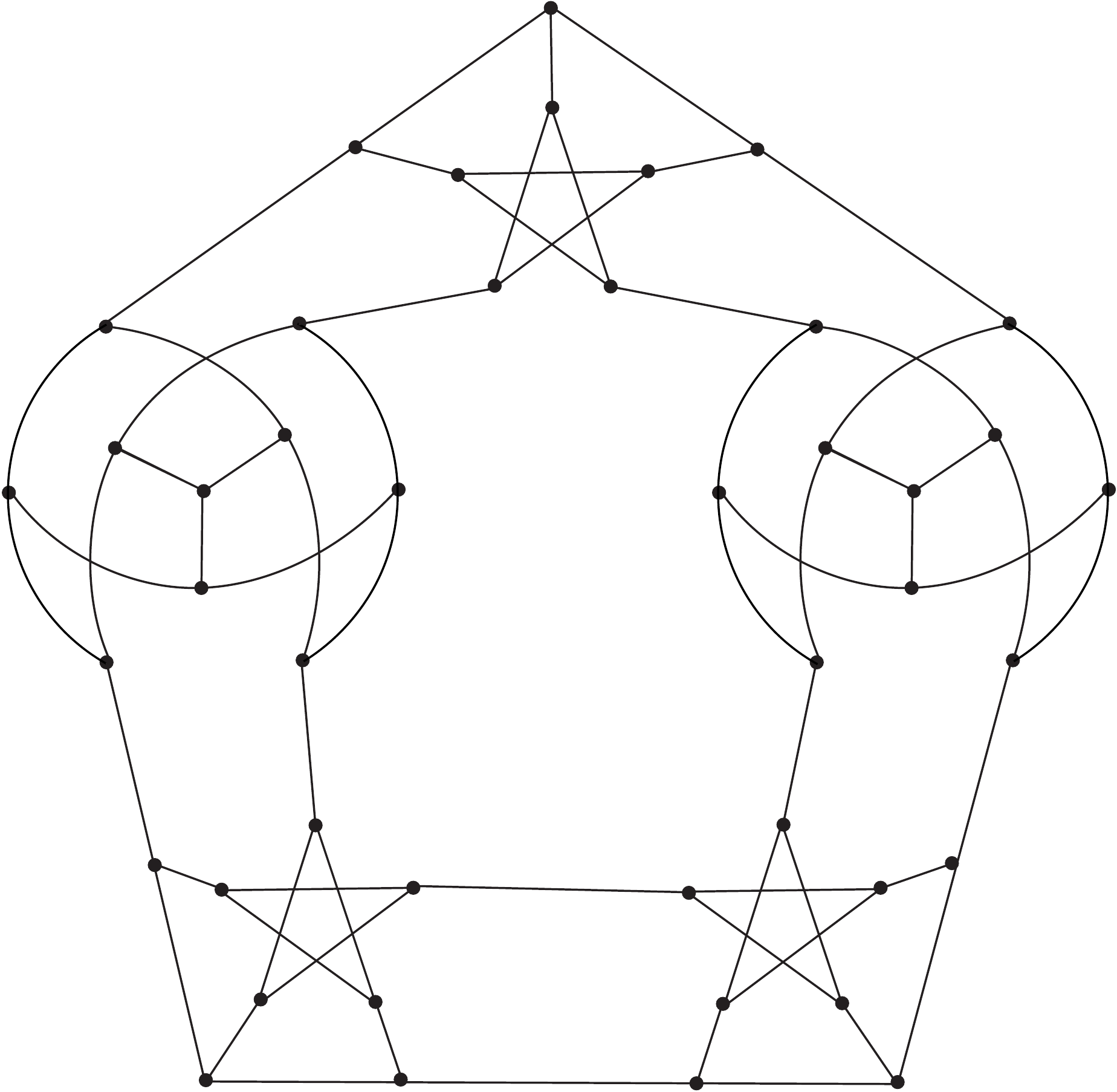}
  \caption{The most symmetric nontrivial snark of oddness $4$ on 44 vertices.}
  \label{fig:sym44}
\end{figure}

As we have already mentioned in Section~\ref{sec:intro} Introduction,
Theorem~\ref{thm:main} does not yet determine the smallest
order of a nontrivial snark with oddness~$4$, because there
might exist snarks with oddness at least $4$ of order 38, 40,
or 42 with cyclic connectivity greater than~$4$. Furthermore,
it is not immediately clear why a snark $G$ with $\omega(G)\ge
m$ and minimum order should have oddness exactly $m$. This
situation suggests two natural problems which require the
following definition: Given integers $\omega\ge 2$ and $k\ge
2$, let $m(\omega,k)$ denote the minimum order of a cyclically
$k$-edge-connected snark with oddness at least~$\omega$. For
example, one has $m(2,2)=m(2,3)=m(2,4)=m(2,5)=10$ as
exemplified by the Petersen graph, and $m(2,6)=28$ as
exemplified by the Isaacs flower snark $J_7$. The values
$m(2,k)$ for $k\ge 7$ are not known, however the well-known
conjecture of Jaeger and Swart~\cite{JS} that there are no
cyclically $7$-edge-connected snarks would imply that these
values are not defined. For $\omega=4$, Lukot\!'ka et
al.~\cite[Theorem~12]{LMMS} showed that $m(4,2)=m(4,3)=28$. The
value $m(4,4)$ remains unknown although our
Theorem~\ref{thm:main} seems to suggest that $m(4,4)=44$.

\bigskip

\noindent\textbf{Problem~1.}
Determine the value $m(4,4)$.

\medskip

Our second problem asks whether the function $m(\omega,k)$ is
monotonous in both coordinates.

\medskip

\noindent\textbf{Problem~2.} Is it true that $m(\omega+1,k)\ge
m(\omega,k)$ and $m(\omega,k+1)\ge m(\omega,k)$ whenever the
involved values are defined?

\medskip

\section{Testing conjectures}

After having generated all snarks from the set
$\mathcal{S}_{34}$ and those from $\mathcal{S}_{36}$ that have
girth at least~$5$, Brinkmann et al.~\cite{BrGHM} tested the
validity of several important conjectures whose minimal
counterexamples, provided that they exist, must be snarks. For
most of the considered conjectures the potential minimal
counterexamples are proven to be nontrivial snarks, that is,
those with cyclic connectivity at least $4$ and girth at
least~$5$. Nevertheless, in some cases the girth condition has
not been established. Therefore it appears reasonable to check
the validity of such conjectures on the set
$\mathcal{S}_{36}$-$\mathcal{S}_{34}$ of all cyclically
$4$-edge-connected snarks of order $36$. We have performed
these tests and arrived at the conclusions discussed below; for
more details on the conjectures we refer the reader
to~\cite{BrGHM}.

\medskip

A \textit{dominating} circuit in a graph $G$ is a circuit $C$
such that every edge of $G$ has an end-vertex on $C$.
Fleischner~\cite{F84} made the following conjecture on
dominating cycles.

\begin{conjecture}[{\rm Dominating circuit conjecture}]
\label{conj:doms}
Every cyclically $4$-edge-connected snark has a dominating
circuit.
\end{conjecture}

The dominating circuit conjecture exists in several different
forms (see, for example, \cite{AJ, FK}) and is equivalent to a
number of other seemingly unrelated conjectures such as the
Matthews-Sumner conjecture about the hamiltonicity of claw-free
graphs \cite{MaSu}. For more information on these conjectures
see \cite{Broe+}.

Our tests have resulted in the following claim.

\begin{claim}
Conjecture~\ref{conj:doms} has no counterexample on $36$ or
fewer vertices.
\end{claim}

\medskip

The \textit{total chromatic number} of a graph $G$ is the
minimum number of colours required to colour the vertices and
the edges of $G$ in such a way that adjacent vertices and edges
have different colours and no vertex has the same colour as its
incident edges. The total colouring conjecture~\cite{BCC, V}
suggests that the total chromatic number of every graph with
maximum degree $\Delta$ is either $\Delta+1$ or $\Delta +2$.
For cubic graphs this conjecture is known to be true by a
result of Rosenfeld~\cite{R71}, therefore the total chromatic
number of a cubic graph is either 4 or 5. Cavicchioli et al.
\cite[Problem 5.1]{C03} asked for a smallest nontrivial snark
with total chromatic number~5. Brinkmann et al.~\cite{BrGHM}
showed that such a snark must have at least $38$ vertices.
Sasaki et al.~\cite{SDFP} displayed examples of snarks with
connectivity $2$ or $3$ whose total chromatic number is $5$ and
asked~\cite[Question~2]{SDFP} for the order of a smallest
cyclically $4$-edge-connected snark with total chromatic number
5. Brinkmann et al.~\cite{BrPS} constructed cyclically
$4$-edge-connected snarks with girth $4$ and total chromatic
number $5$ for each even order greater than or equal to~$40$.
Our next claim shows that the value asked for by Sasaki et al.\
is either 38 or 40.

\begin{claim}
All cyclically $4$-edge-connected snarks with at most $36$
vertices have total chromatic number~$4$.
\end{claim}

\medskip

The following conjecture was made by Jaeger~\cite{J88} and is
known as the \textit{Petersen colouring conjecture}. If true,
this conjecture would imply several other profound
conjectures, in particular, the $5$-cycle double cover
conjecture and the Fulkerson conjecture.

\begin{conjecture}[{\rm Petersen colouring conjecture}]
\label{conj:petersen_col}
Every bridgeless cubic graph $G$ admits a colouring of its
edges using the edges of the Petersen graph as colours in such
a way that any three mutually adjacent edges of $G$ are
coloured with three mutually adjacent edges of the Petersen
graph.
\end{conjecture}

It is easy to see that the smallest counterexample to this
conjecture must be a cyclically $4$-edge-connected snark.
Brinkmann et al.~\cite{BrGHM} showed that the smallest
counterexample to the Petersen colouring conjecture must have
order at least $36$. Here we improve the latter value to $38$.

\begin{claim}
Conjecture~\ref{conj:petersen_col} has no counterexamples on
$36$ or fewer vertices.
\end{claim}

\subsection*{Acknowledgements}
The first author was supported by a Postdoctoral Fellowship of
the Research Foundation Flanders (FWO). The second and the
third author were partially supported by VEGA 1/0876/16 and by
APVV-15-0220. Most computations for this work were carried out
using the Stevin Supercomputer Infrastructure at Ghent
University.

\end{document}

%% file: figures/sets.tex
\begin{picture}(0,0)%
\includegraphics{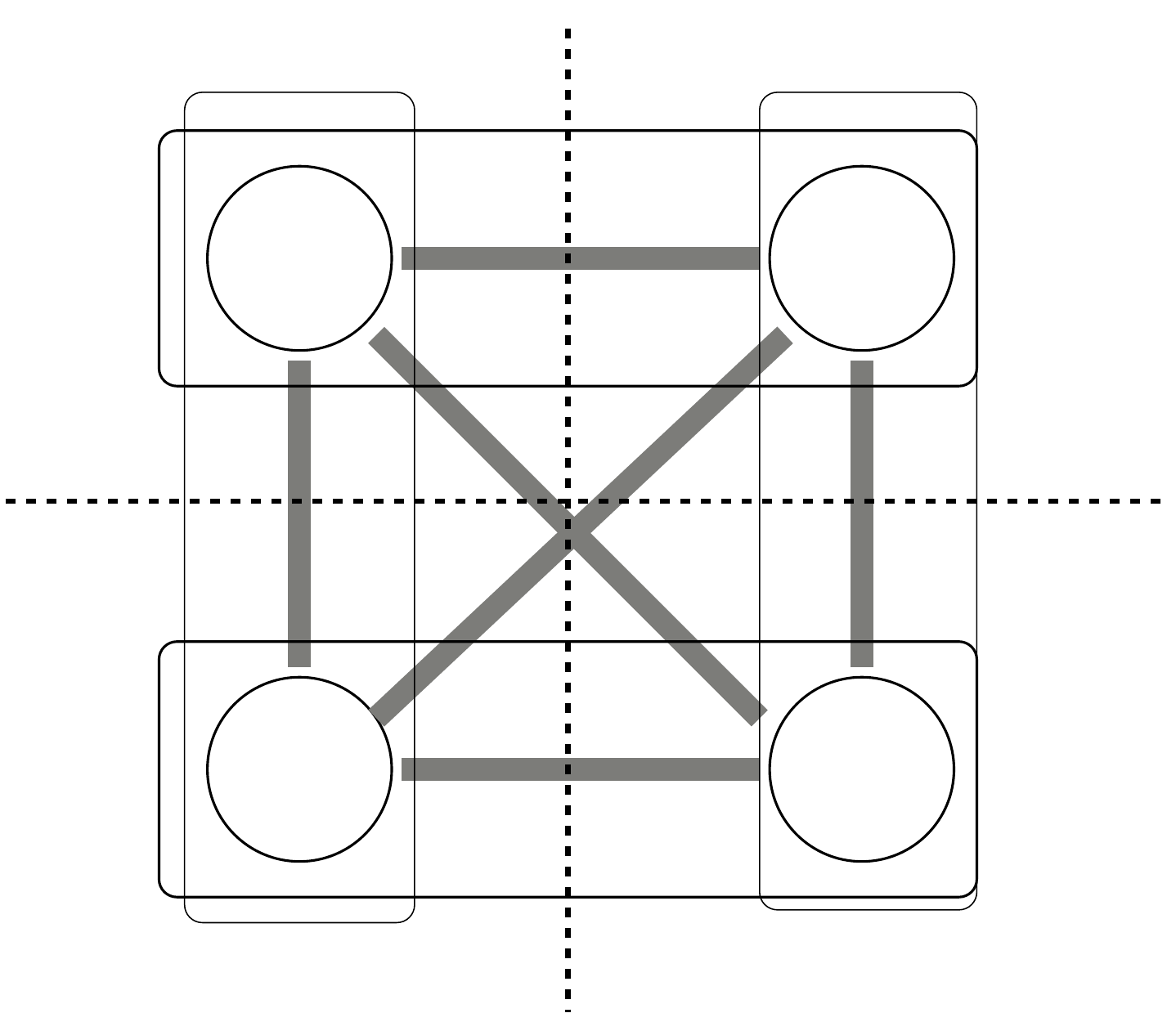}%
\end{picture}%
\setlength{\unitlength}{3947sp}%
\begingroup\makeatletter\ifx\SetFigFont\undefined%
\gdef\SetFigFont#1#2#3#4#5{%
  \reset@font\fontsize{#1}{#2pt}%
  \fontfamily{#3}\fontseries{#4}\fontshape{#5}%
  \selectfont}%
\fi\endgroup%
\begin{picture}(6891,6037)(643,-4895)
\put(5401,839){\makebox(0,0)[lb]{\smash{{\SetFigFont{20}{24.0}{\rmdefault}{\mddefault}{\updefault}{\color[rgb]{0,0,0}$Y_2$}%
}}}}
\put(901,-3511){\makebox(0,0)[lb]{\smash{{\SetFigFont{20}{24.0}{\rmdefault}{\mddefault}{\updefault}{\color[rgb]{0,0,0}$X_2$}%
}}}}
\put(901,-511){\makebox(0,0)[lb]{\smash{{\SetFigFont{20}{24.0}{\rmdefault}{\mddefault}{\updefault}{\color[rgb]{0,0,0}$X_1$}%
}}}}
\put(3301,-211){\makebox(0,0)[lb]{\smash{{\SetFigFont{20}{24.0}{\rmdefault}{\mddefault}{\updefault}{\color[rgb]{0,0,0}$a$}%
}}}}
\put(3526,-1336){\makebox(0,0)[lb]{\smash{{\SetFigFont{20}{24.0}{\rmdefault}{\mddefault}{\updefault}{\color[rgb]{0,0,0}$b$}%
}}}}
\put(2551,-2236){\makebox(0,0)[lb]{\smash{{\SetFigFont{20}{24.0}{\rmdefault}{\mddefault}{\updefault}{\color[rgb]{0,0,0}$c$}%
}}}}
\put(4276,-3736){\makebox(0,0)[lb]{\smash{{\SetFigFont{20}{24.0}{\rmdefault}{\mddefault}{\updefault}{\color[rgb]{0,0,0}$e$}%
}}}}
\put(5851,-2311){\makebox(0,0)[lb]{\smash{{\SetFigFont{20}{24.0}{\rmdefault}{\mddefault}{\updefault}{\color[rgb]{0,0,0}$f$}%
}}}}
\put(7276,-2236){\makebox(0,0)[lb]{\smash{{\SetFigFont{20}{24.0}{\rmdefault}{\mddefault}{\updefault}{\color[rgb]{0,0,0}$R$}%
}}}}
\put(4126,-4786){\makebox(0,0)[lb]{\smash{{\SetFigFont{20}{24.0}{\rmdefault}{\mddefault}{\updefault}{\color[rgb]{0,0,0}$T$}%
}}}}
\put(4801,-1561){\makebox(0,0)[lb]{\smash{{\SetFigFont{20}{24.0}{\rmdefault}{\mddefault}{\updefault}{\color[rgb]{0,0,0}$d$}%
}}}}
\put(2176,839){\makebox(0,0)[lb]{\smash{{\SetFigFont{20}{24.0}{\rmdefault}{\mddefault}{\updefault}{\color[rgb]{0,0,0}$Y_1$}%
}}}}
\end{picture}%

%% file: figures/closure2.tex
\begin{picture}(0,0)%
\includegraphics{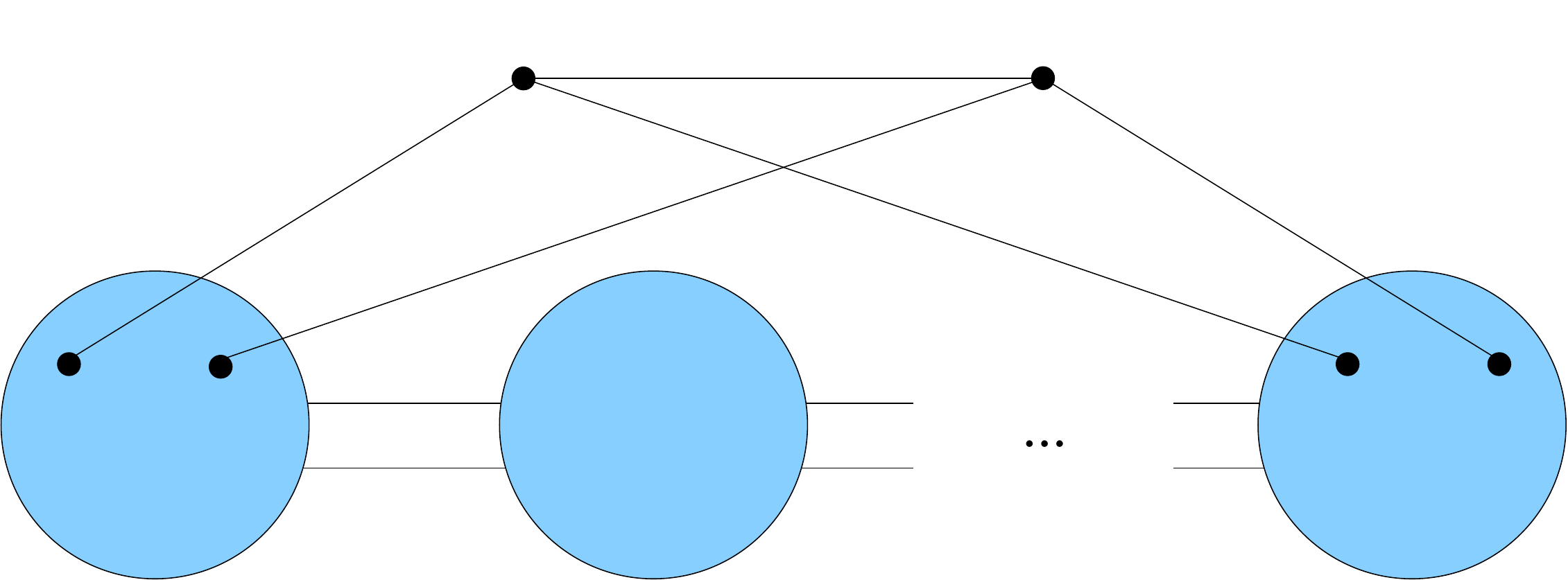}%
\end{picture}%
\setlength{\unitlength}{3947sp}%
\begingroup\makeatletter\ifx\SetFigFont\undefined%
\gdef\SetFigFont#1#2#3#4#5{%
  \reset@font\fontsize{#1}{#2pt}%
  \fontfamily{#3}\fontseries{#4}\fontshape{#5}%
  \selectfont}%
\fi\endgroup%
\begin{picture}(10848,4001)(727,2316)
\put(7801,6014){\makebox(0,0)[lb]{\smash{{\SetFigFont{20}{24.0}{\rmdefault}{\mddefault}{\updefault}$x_2$}}}}
\put(1051,3389){\makebox(0,0)[lb]{\smash{{\SetFigFont{20}{24.0}{\rmdefault}{\mddefault}{\updefault}$a_1$}}}}
\put(9901,3389){\makebox(0,0)[lb]{\smash{{\SetFigFont{20}{24.0}{\rmdefault}{\mddefault}{\updefault}$a_3$}}}}
\put(2101,3389){\makebox(0,0)[lb]{\smash{{\SetFigFont{20}{24.0}{\rmdefault}{\mddefault}{\updefault}$a_2$}}}}
\put(10951,3389){\makebox(0,0)[lb]{\smash{{\SetFigFont{20}{24.0}{\rmdefault}{\mddefault}{\updefault}$a_4$}}}}
\put(4201,6014){\makebox(0,0)[lb]{\smash{{\SetFigFont{20}{24.0}{\rmdefault}{\mddefault}{\updefault}$x_1$}}}}
\end{picture}%